\documentclass[11pt, reqno]{amsart}
\usepackage{amsaddr}
\usepackage[latin1]{inputenc}
\usepackage{amsmath, amssymb}
\usepackage{centernot}
\usepackage{dsfont}
\usepackage{bbm}
\usepackage{float}
\usepackage[shortlabels]{enumitem}
\usepackage{graphicx}                  % derni\`ere \'etant la langue principale
\usepackage{amssymb}
\usepackage{amsfonts}
\RequirePackage{amsmath}
\RequirePackage{amsthm}

\usepackage{color}
\usepackage{dsfont}
\usepackage{geometry}
\usepackage{fancyhdr}
\usepackage{lipsum} 
\newtheorem{theorem}{Theorem}[section]
\newtheorem*{thm*}{Theorem}
\newtheorem{proposition}[theorem]{Proposition}
\newtheorem{lem}[theorem]{Lemma}
\newtheorem{cor}[theorem]{Corollary}

\usepackage{hyperref} 

\theoremstyle{definition}
\newtheorem{definition}[theorem]{Definition}
\newtheorem{remark}{Remark}[section]
\newtheorem{example}{Example}[section]

\DeclareMathOperator{\Var}{Var}

\DeclareMathOperator{\Ent}{Ent}

\DeclareMathOperator{\Ric}{Ric}
\providecommand{\abs}[1]{\left|#1\right|}
\providecommand{\norm}[1]{\left\lVert#1\right\rVert}
\providecommand{\hs}[1]{\left|#1\right|_{\mathrm{HS}}}

\newcommand{\R}{\mathbb{R}}
\newcommand{\N}{\mathbb{N}}

\newcommand{\C}{\mathcal{C}}

\renewcommand{\phi}{\varphi}
\renewcommand{\epsilon}{\varepsilon}
\newcommand{\1}{\mathds{1}}

\renewcommand{\P}{\mathrm{P}}
\renewcommand{\L}{\mathrm{L}}
\newcommand{\Div}{\nabla \cdot }

\title{A Bakry-\'Emery approach to Lipschitz transportation on manifolds}
\author{Pablo López-Rivera}
\date{\today}
\address{Université Paris Cité \& Sorbonne Université, CNRS, Laboratoire Jacques-Louis Lions (LJLL), F-75013 Paris, France}
\email{plopez@math.univ-paris-diderot.fr}
\calclayout
\begin{document}

\maketitle

\begin{abstract}
On weighted Riemannian manifolds we prove the existence of globally Lipschitz transport maps between the weight (probability) measure and log-Lipschitz perturbations of it, via Kim and Milman's diffusion transport map, assuming that the curvature-dimension condition $\mathrm{CD}(\rho_{1}, \infty)$ holds, as well as a second order version of it, namely $\Gamma_{3} \geq \rho_{2} \Gamma_{2}$. We get new results as corollaries to this result, as the preservation of Poincar\'e's inequality for the exponential measure on $(0,+\infty)$ when perturbed by a log-Lipschitz potential and a new growth estimate for the Monge map pushing forward the gamma distribution on $(0,+\infty)$ (then getting as a particular case the exponential one), via Laguerre's generator.
\end{abstract}

%%%%% SECTION: Introduction
\section{Introduction}
\label{section1}
In the study of functional inequalities, the $d$-dimensional Gaussian space, $(\R^{d}, \gamma_{d})$, where $\gamma_{d}$ denotes the standard $d$-dimensional Gaussian measure, has been of central importance, as $\gamma_{d}$ satisfies fundamental model inequalities, as Poincar\'e's inequality: for each $f \colon \R^{d} \to \R$ smooth and square-integrable,
$$\Var_{\gamma_{d}}(f) \leq \mathbb{E}_{\gamma_{d}}[\abs{\nabla f}^{2}].$$
Given a second measure $\nu$ on $\R^{d}$ and a $C$-Lipschitz map $T\colon \R^{d} \to \R^{d}$ such that $T$ pushes forward $\gamma_{d}$ towards $\nu$, we note that it is easy to transport as well Gaussian Poincar\'e's inequality to $\nu$: let $g \colon \R^{d} \to \R$ be a smooth function and let us apply Gaussian Poincar\'e's inequality to $f = g \circ T$:
\begin{equation}
\label{eqfirstexample}
\Var_{\nu}(g) = \Var_{\gamma_{d}}(g\circ T) \leq \mathbb{E}_{\gamma_{d}}[\abs{\nabla (g\circ T)}^{2}] \leq \mathbb{E}_{\gamma_{d}}[\abs{DT}^{2} \abs{(\nabla g)\circ T}^{2}] \leq C^{2} \mathbb{E}_{\nu}[\abs{\nabla g}^{2}],
\end{equation}
where in the last inequality we stressed the fact that $T$ is a $C$-Lipschitz map. We remark that the same argument applies to log-Sobolev or isoperimetric inequalities.

In \cite{MR1800860} Caffarelli proved his famous contraction theorem, which states that for a measure $\nu$ on $\R^{d}$ which is uniformly log-concave (i.e., $\nu$ is of the form $d\nu(x) = \exp(-V)\, dx$ with $V\colon \R^{d} \to \R$ and such that there exists $\alpha >0$ such that for each $x \in \R^{d}$, $\nabla^{2} V(x) \succcurlyeq \alpha I_{d \times d}$, where $\succcurlyeq$ denotes the L\"owner order on the set of positive semidefinite matrices), then the Brenier map \cite{MR1100809} (i.e., the transport map solving Monge's problem; see \cite{villani2008optimal} for more context on Monge's problem), which pushes forward $\gamma_{d}$ towards $\nu$, is $C$-Lipschitz, with a constant $C$ which does not depend on the dimension $d$. Thus Caffarelli's regularity theorem plays an important role when one wants to obtain novel inequalities; not just functional but concentration and isoperimetric ones, which are still independent of the dimension parameter $d$ (see for example \cite{MR1894593} and \cite{MR2095937}).

There are some extensions to Caffarelli's theorem: for example in \cite{MR3752535} it was proven that the Brenier map between a uniformly log-concave measure and a compactly supported perturbation of it is Lipschitz, again with a constant not depending on the dimension. Unfortunately, we remark that there are no known analogs for Caffarelli's result for the Riemannian setting.

As it was seen in the transport of the Gaussian Poincar\'e's inequality (\ref{eqfirstexample}), optimality of the transport map ---in the sense of Monge's problem--- does not play any role in the proof. In particular, in this note we prove the existence of a globally Lipschitz transport map (not necessarily optimal in Monge's sense) between a measure on a manifold and a log-Lipschitz perturbation of it, using a construction originally developed in \cite{kimmilman2012} (and which traces its origins back to \cite{otto2000generalization}), revisited as well recently in \cite{klartag2021spectral}, \cite{neeman2022lipschitz}, \cite{mikulincer2022lipschitz} and \cite{fathi2023transportation}, which we will call the diffusion transport map, as it is based on the heat flow interpolation between two measures induced by a diffusion semigroup. Similar schemes, using Polchinski's flow and multiscale Bakry-\'Emery criteria, have been recently explored in \cite{serres2022behavior} and \cite{shenfeld2023exact}, for example.

In our approach, we consider a diffusion operator $\L$ on a Riemannian manifold $(M,g)$ which has an ergodic measure $\mu$. Fix a $K$-Lipschitz potential $V \colon M \to \R$ and define the perturbation of $\mu$ by $V$ as the measure $d\nu = e^{-V}\, d\mu$. Using the aforementioned construction, we are able to obtain a map $T \colon M \to M$ pushing forward the measure $\mu$ towards $\nu$.

Defining the carr\'e du champ operator $\Gamma$, its iteration $\Gamma_{2}$ and its second iteration $\Gamma_{3}$ (which will be precisely defined in Section \ref{section3}), if we suppose that there exist positive constants $\rho_{1}, \rho_{2}>0$ such that for each smooth and compactly supported function $f \colon M \to \R$, $\Gamma_{2}(f) \geq \rho_{1} \Gamma(f)$ (i.e, classical Bakry-\'Emery's curvature-dimension condition $\mathrm{CD}(\rho, \infty)$) and $\Gamma_{3}(f) \geq \rho_{2} \Gamma_{2}(f)$, then we are able to say that $T$ is $C$-Lipschitz, where $C = C(\rho_{1}, \rho_{2}, K)$. The following theorem is the precise statement of these ideas, corresponding to Theorem \ref{thmmainresult} in this note.

\begin{thm*}
Let $(M,g)$ be a complete and connected $d$-dimensional Riemannian manifold. Let $W \colon M \to \R$ be smooth and such that $\int_{M} e^{-W}d \mathrm{Vol} =1$. Consider the diffusion operator $\L=\Delta - \nabla W \cdot \nabla$. Let us assume that there exist constants $\rho_{1}, \rho_{2} > 0$ such that
	\begin{enumerate}[(i)]
	\item $\forall f \in \C^{\infty}_{\mathrm{c}}(M), \Gamma_{2}(f) \geq \rho_{1} \Gamma(f)$; and
	\item $\forall f \in \C^{\infty}_{\mathrm{c}}(M), \Gamma_{3}(f) \geq \rho_{2} \Gamma_{2}(f)$.
	\end{enumerate}
Let $V \colon M\to \R$ smooth and $K$-Lipschitz and define $d \mu = e^{-W}d \mathrm{Vol}$ and $d \nu = e^{-V} d\mu$. Then there exists a Lipschitz mapping $T \colon M \to M$ pushing forward $\mu$ towards $\nu$ which is $\exp\left(\sqrt{\frac{2 \pi}{\rho_{2}}}Ke^{\frac{K^{2}}{2 \rho_{1}}}\right)$-Lipschitz.
\end{thm*}

The idea of working with $\Gamma_{3}$ is not new in the literature; it has been explored previously in \cite{zbMATH00559130}, \cite{zbMATH00563664} and \cite{MR3320893}. In particular, in \cite{MR3320893}, the condition $\Gamma_{3}(f) \geq \rho_{2} \Gamma_{2}(f)$ was employed to obtain regularity of the solutions of the PDE
$$\Delta f - \nabla W \cdot \nabla f = V - \int V \, d\mu$$
for Lipschitz data $V$, which corresponds to the linearization of Monge-Amp\`ere's equation associated to the (quadratic) optimal transport problem for $\mu$ and $\nu$, which coincides as well with the linearization of the diffusion transport map.

\subsection{Organization}
In Section \ref{section2} we introduce all the necessary preliminary notions, starting from diffusion operators on Riemannian manifolds, their probabilistic interpretation and the continuity equation in the distributional sense. In Section \ref{section3} we recall classical Bakry-\'Emery's $\Gamma$-calculus and we study a second order version of it, namely $\Gamma_{3} \geq \rho \Gamma_{2}$, and its consequences. In Section \ref{section4} we detail the construction of the diffusion transport map and we prove its Lipschitz regularity (i.e. Theorem \ref{thmmainresult}). Finally, in Section \ref{section5} we give some applications of the main result.

\subsection{Acknowledgments}
I want to thank my advisor, Max Fathi, for his guidance, corrections and continuous support with this project, and for introducing me to the subject. I also thank Maxime Laborde for his help. The author acknowledges the support of the French Agence Nationale de la Recherche (ANR) via the project CONVIVIALITY (ANR-23-CE40-0003). This project has received funding from the European Union's Horizon 2020 research and innovation programme under the Marie Sk\l{}odowska-Curie grant agreement No 945322. \includegraphics[scale=0.030]{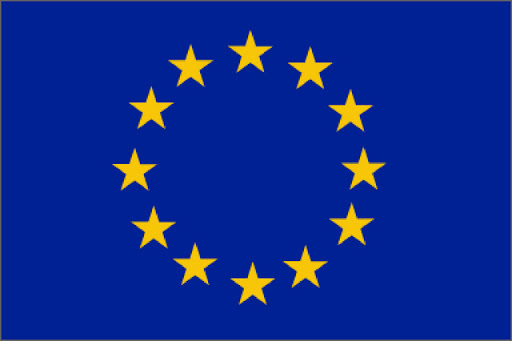}

%%%%% SECTION: Preliminaries and notations
\section{Preliminaries and notations}
\label{section2}
In this section we present all the main objects on which this note is based and their most essential properties; on the other hand we fix the notations for the rest of this note. We first introduce the language of Riemannian geometry in order to talk about diffusion operators on weighted manifolds, to then remind as well the probabilistic interpretation of these objects. Finally we recall the continuity equation in the context of measure-valued flows and the structure of its solutions.

\subsection{Markov diffusion generators on a manifold}

Let $(M,g)$ be a $d$-dimensional smooth Riemannian manifold which we assume to be complete and connected, unless otherwise stated, and let us denote its tangent bundle by $TM$. Usually the action of the metric $g$ onto two vector fields $X, Y \colon M \to TM$ will be succinctly written by $X\cdot Y := g(X,Y)$, omitting the dependence on $x \in M$, unless necessary. The metric $g$ induces a distance $d_{g}$ on $M$, which in turn generates the Riemannian volume measure $d\mathrm{Vol}$ on $M$ (i.e.,  the $d$-dimensional Hausdorff measure associated to the metric space $(M, d_{g})$).

Let $\nabla$ be the Riemannian gradient, which acts on smooth functions $f \colon M\to \R$ as a vector field by $\nabla f =(\nabla^{i} f)_{i=1}^{d}$, where $\nabla^{i}f = g^{ij} \partial_{j}f$ (using Einstein's summation convention to handle operations with tensors) and let $\nabla^{2}$ be the Hessian operator, which acts on smooth functions $f \colon M \to \R$ via $\nabla^{2} f = (\nabla^{i}\nabla^{j} f)_{i,j=1}^{d}$. Laplace-Beltrami's operator, denoted by $\Delta$, acts on smooth functions and this action can be characterized, for a smooth function $f\colon M \to \R$, by
$$\forall g \in \C^{\infty}_{\mathrm{c}}(M), \quad \int_{M} g \Delta f \, d\mathrm{Vol} = - \int_{M} \nabla f \cdot \nabla g \,d\mathrm{Vol},$$
where $\C^{\infty}_{\mathrm{c}}(M)$ is the set of real-valued smooth functions with compact support on $M$. We recall Bochner's formula: for each $f\colon M \to \R$ smooth,
\begin{equation}
\label{eqbochnerformula}
\frac12 \Delta(\abs{\nabla f}^{2}) = \abs{\nabla^{2} f}_{\mathrm{HS}}^{2} + \nabla f \cdot \nabla(\Delta f) + \Ric(\nabla f, \nabla f),
\end{equation}
where $\abs{\cdot}_{\mathrm{HS}}$ denotes the Hilbert-Schmidt norm of a tensor and $\Ric$ denotes the Ricci curvature tensor associated to $(M,g)$.

Now let $W \colon M \to \R$ be a smooth function such that $\int_{M} e^{-W}\, d\mathrm{Vol}=1$ to then define the (probability) measure $d\mu(x) := e^{-W} \, d\mathrm{Vol}$. We say that $(M,g,\mu)$ is a weighted manifold (that is, a Riemannian manifold equipped with a measure which has a smooth strictly positive density with respect to the respective volume measure).

We define the diffusion operator $\L$, which acts on smooth functions $f \colon M \to \R$ as
\begin{equation}
\label{eqoperatorldef}
\L f:= \Delta f - \nabla W \cdot \nabla f,
\end{equation}
which obviously has elliptic regularity. We note that the measure $\mu$ is invariant for $\L$: for each $f \in \C^{\infty}_{\mathrm{c}}(M)$,
\begin{equation}
\label{eqinvariantmeasure}
\int_{M}\L f \, d\mu =0.
\end{equation}
Moreover, it can be easily seen that it holds the following integration by parts formula: for all $f, h \in \C^{\infty}_{\mathrm{c}}(M)$,
\begin{equation}
\label{eqsymmetryipp}
\int_{M}f \L h\, d\mu = - \int_{M}\nabla f \cdot \nabla h\, d\mu,
\end{equation}
which in turn implies that $\L$ is symmetric or reversible with respect to the measure $\mu$: that is, for each $f, h \in \C^{\infty}_{\mathrm{c}}(M)$,
\begin{equation}
\label{eqdefsymmetry}
\int_{M}f \L h\, d\mu =  \int_{M}h \L f\, d\mu.
\end{equation}

\begin{remark}
\label{rkmuchotexto}
Suppose now that we have a diffusion operator $\L$ defined on $M$, a connected $d$-dimensional smooth manifold (note that we are not endowing $M$ with any Riemannian metric), which acts on smooth functions $f \colon M \to \R$ (in local coordinates) via 
\begin{equation}
\label{eqdefl}
\L f := a^{ij} \partial_{ij}^{2} f +b^{i}\partial_{i} f,
\end{equation}
such that the coefficients $x \mapsto a(x):= (a^{ij}(x))_{i,j =1}^{d}$, $x \mapsto b(x):= (b^{i}(x))_{i=1}^{d}$ are smooth functions on $M$. Let us assume that $\L$ is elliptic: that is, for each $x \in M$, the matrix $a(x)$ is symmetric and positive-definite; then the matrix $a^{-1} = (a_{ij}(x))_{i,j =1}^{d}$ can be regarded as a Riemannian metric on $M$ which we will denote by $g$. If we write $Z := b^{i} \partial_{i} \colon M \to TM$, then we can express $\L$ as
$$\L f = \Delta_{g} f + Zf,$$
where $\Delta_{g}$ denotes the Laplace-Beltrami operator on the (now) Riemannian manifold $(M, g)$. Moreover, let us suppose that $\L$ is invariant with respect to a measure $\mu$ of the form $d\mu = e^{-W}\, d\mathrm{Vol}_{g}$, where $\mathrm{Vol}_{g}$ denotes the Riemannian volume measure on $(M,g)$ and $W\colon M \to \R$ is a smooth function; then $Zf = -\nabla_{g} W \cdot \nabla_{g} f$, so finally we get
$$\L f = \Delta_{g} f - \nabla_{g} W \cdot \nabla_{g} f.$$
That is, if we start with a linear diffusion operator in the form (\ref{eqdefl}), we can always express it in the form (\ref{eqoperatorldef}), for a well-chosen Riemannian structure, which is natural for the operator $\L$. In other words, we can assume without loss of generality that $\L=\Delta - \nabla W \cdot \nabla$.
\end{remark}

From the perspective of operator theory, $\L$ could be regarded as well as an unbounded linear operator with domain $\mathcal{D}(\L) \subseteq L^{2}(\mu)$ such that $\C^{\infty}_{\mathrm{c}}(M) \subseteq \mathcal{D}(\L)$. Moreover, as we assumed that $(M,g)$ is complete, as $\L$ is an elliptic diffusion operator on a smooth, connected and complete Riemannian manifold, which is symmetric with respect to $\mu$, then Proposition 3.2.1 in \cite{bakry2013analysis} entails the essential self-adjointness of $\L$ on $\C^{\infty}_{\mathrm{c}}(M)$, or equivalently, the density of $\C^{\infty}_{\mathrm{c}}(M)$ into $\mathcal{D}(\L)$.

\begin{remark}
Following the discussion about the essential self-adjointness of $\L$, we point out two possible issues that could happen in the practice.
	\begin{itemize}
	\item First, sometimes we will not have the hypothesis of completeness of the manifold $(M,g)$ (this is the case when $M$ is an open ball or more generally an open connected domain $\mathcal{O} \neq \R^{d}$, both regarded as submanifolds of the Euclidean space $\R^{d}$; more generally, this is an issue which depends on the metric chosen on the space). Fortunately, essential self-adjointness is not an exclusive property of diffusion operators on complete manifolds.
	\item On the other hand, it could happen that the operator $\L$ is essentially self-adjoint, but with respect to another class of functions, not $\C^{\infty}_{\mathrm{c}}(M)$.
	\end{itemize}
\end{remark}

Following Hille-Yosida theory \cite{yosida1980functional}, if we suppose that $\L$ is essentially self-adjoint with respect to $\C^{\infty}_{\mathrm{c}}(M)$ (or another class of smooth functions dense in $L^{2}(\mu)$), then there exists a unique characteristic semigroup $(\P_{t})_{t \geq 0}$ on $L^{2}(\mu)$ associated to $\L$ solving the heat equation for $\L$: for each $f \in L^{2}(\mu)$, then $u(t,x) := \P_{t}f(x)$ is the unique solution of
\begin{equation}
\label{eqheatequationl}
\left\{
		\begin{aligned}
		\partial_{t} u(t,x) & = \L u(t,x)\\
		u(0,x)& = f(x);
		\end{aligned}
	\right.
\end{equation}
in other words, $\L$ corresponds to the infinitesimal generator of the semigroup $(\P_{t})_{t \geq 0}$.

We defined our framework around the infinitesimal generator $\L = \Delta - \nabla W \cdot \nabla$, but it is possible to translate these properties, in an equivalent way, into the language of the semigroup $(\P_{t})_{t \geq 0}$. The invariance (\ref{eqinvariantmeasure}) reads as having for each $f \in \C^{\infty}_{\mathrm{c}}(M)$, and any $t \geq 0$,
\begin{equation*}
\label{eqinvariantmeasuresemigroup}
\int_{M}\P_{t}f \, d\mu = \int_{M} f\, d\mu;
\end{equation*}
and the symmetry (\ref{eqdefsymmetry}) is the same as stating that for each $f, h \in \C^{\infty}_{\mathrm{c}}(M)$ and any $t \geq 0$,
\begin{equation*}
\label{eqdefsymmetrysemigroup}
\int_{M}f \P_{t}h\, d\mu =  \int_{M}h \P_{t}f\, d\mu.
\end{equation*}

Finally, we can see that $\L = \Delta - \nabla W \cdot \nabla$ is ergodic; i.e., if $f \in \C^{\infty}_{\mathrm{c}}(E)$ and $\L f =0$, then $f$ has to be identically constant. Indeed, if we suppose that $\L f =0$, then by (\ref{eqsymmetryipp}) we deduce that $\abs{\nabla f}^{2}=0$ in $M$, which in turn yields that $f$ is constant, as $M$ was assumed to be connected. The following proposition translates ergodicity into a property for the semigroup $(\P_{t})_{t \geq 0}$ (see for example \cite[Proposition 3.1.13]{bakry2013analysis}).

\begin{proposition}
\label{propergodicityl}
Suppose that $\L$ is ergodic and that the corresponding semigroup $(\P_{t})_{t \geq 0}$ is conservative, that is, for each $t \geq 0$, $\P_{t} \1 = \1$, where $\1$ denotes the constant function equal to 1 on $M$. If $\mu$ is a probability measure, then for each $f \in L^{2}(\mu)$, the following limit holds in the $L^{2}(\mu)$-sense:
$$\lim_{t \to \infty} \P_{t} f = \int_{E} f\, d\mu.$$
\end{proposition}

%%%%% SUBSECTION: The probabilistic counterpart of $\L$
\subsection{The probabilistic counterpart of $\L$}

As we were able (via (\ref{eqheatequationl})) to construct the characteristic semigroup associated to $\L$, then we can study $\L$ using probabilistic techniques. First of all, for any $t \geq 0$, we denote by $\P_{t}^{*}$ the dual semigroup, which acts on a measure $\nu \in \mathcal{P}$ by
$$\forall f \in \C_{\mathrm{b}}(M), \quad \int_{M}f \, d(\P_{t}^{*} \nu) =  \int_{M} \P_{t} f \, d\nu,$$
where $ \C_{\mathrm{b}}(M)$ denotes the space of bounded and continuous real-valued functions on $(M, g)$. Now let $(X_{t})_{t \geq 0}$ be a Markov process (defined on a probability space $(\Omega, \mathcal{F}, \mathbb{P})$) such that $X_{0} \sim \nu$ and that for each $t \geq 0$,
\begin{equation}
\label{eqlawdiffusion}
\forall f \in \C_{\mathrm{b}}(M), \quad \mathbb{E}\left[f(X_{t})\right] = \int_{M} f \, d(\P_{t}^{*} \nu).
\end{equation}
In particular, if $x \in M$ and $\nu = \delta_{x}$ is the Dirac measure on $x$, then $\mathbb{E}\left[f(X_{t})\right] = \P_{t}f(x)$.

The Markov process $(X_{t})_{t \geq 0}$ defined above is a diffusion process (with random initial condition $\nu \in \mathcal{P}(M)$) on $M$ satisfying the following stochastic differential equation (SDE):
\begin{equation}
\label{eqsdediffusion}
\left\{
	\begin{aligned}
	d X_{t} &= \sqrt{2}\, dB_{t} - \nabla W(X_{t})\, dt\\
	X_{0} &\sim \nu,
	\end{aligned}
\right.
\end{equation}
where $(B_{t})_{t \geq 0}$ is the Brownian motion on $(M,g)$ (see for example \cite{MR1030543} for a succinct introduction to the language of stochastic calculus on manifolds).

If we write, for each $t \geq 0$, $\rho_{t} := \mathrm{Law}(X_{t})$ (in particular, $\rho_{0} = \nu$), then $(\rho_{t})_{t \geq 0}$ satisfies the Fokker-Planck equation (which is adjoint of the heat equation (\ref{eqheatequationl}))
\begin{equation}
\label{eqfokkerplanck}
\partial_{t} \rho_{t} = \L^{*} \rho_{t} = \Delta \rho_{t} + \nabla W \cdot \nabla \rho_{t} + \rho_{t} \Delta W
\end{equation}
in the distributional sense; i.e., for all $t >0$ and each smooth function with compact support $\varphi\colon M\to \R$,
$$\frac{d}{dt} \int_{M} \phi \, d\rho_{t} =  \int_{M} \L \phi \, d\rho_{t} = \int_{M} \left(\Delta \phi - \nabla W \cdot \nabla \phi \right)\, d\rho_{t}.$$

%%%%% SUBSECTION: About the continuity equation
\subsection{About the continuity equation}
Fix $A_{\bullet}\colon [0,+\infty) \times M \to TM$ a time-depending vector field on $M$. Given a fixed probability measure $\nu$ on $M$, we are interested in solving the continuity equation on $M$, with velocity $A_{\bullet}$ and initial condition $\nu$,
\begin{equation}
	\label{eqpropcontinuityequation}
	\left\{
		\begin{aligned}
		\partial_{t} \rho_{t} + \Div (\rho_{t} A_{t})&=0, \quad t >0\\
		\rho_{0} &= \nu,
		\end{aligned}
	\right.
\end{equation}
where $\Div$ denotes the divergence operator on $(M,g)$, which acts on a vector field $Z\colon M \to TM$ by
$$\forall \phi \in \C^{\infty}_{\mathrm{c}}(M), \quad \int_{M}\phi \Div Z \, d \mathrm{Vol} = - \int_{M} \nabla \phi \cdot Z \, d \mathrm{Vol}.$$

We will be interested in measure-valued solutions for (\ref{eqpropcontinuityequation}). To be precise, we say that $(\rho_{t})_{t \geq 0}$, a sequence of probability measures on $M$ with $\rho_{0} = \nu$, satisfies the continuity equation (\ref{eqpropcontinuityequation}) in the distributional sense, if for any $t >0$ and each smooth function with compact support $\varphi\colon M\to \R$,
\begin{equation*}
\frac{d}{dt} \int_{M} \phi \, d\rho_{t} = \int_{M} \nabla \varphi \cdot A_{t} \, d\rho_{t}.
\end{equation*}

Now we introduce the flow of diffeomorphisms $(S_{t})_{t \geq 0}$ induced by the vector field $A_{\bullet}$, which is defined by
\begin{equation}
\label{eqodeflowat}
\left\{
		\begin{aligned}
		\frac{d}{dt}S_{t}(x) &= A_{t}(S_{t}(x)), \quad t >0\\
		S_{0}(x)&=x,
		\end{aligned}
	\right.
\end{equation}
which is well-defined if $A_{\bullet}\colon$ is at least locally Lipschitz. The following classic result (see for example \cite[Chapter 1]{villani2008optimal}) characterizes weak solutions of the system (\ref{eqpropcontinuityequation}) in terms of the associated flow $(S_{t})_{t \geq 0}$.

\begin{theorem}
\label{thmcontinuityequation}
Let $(M,g)$ be a complete and connected Riemannian manifold and fix $A_{\bullet}\colon [0,+\infty) \times M \to TM$ a locally Lipschitz time-depending vector field on $M$. Let $\nu$ be a probability measure on $M$ and let $(\rho_{t})_{t \geq 0}$ be a sequence of probability measures on $M$ such that $\rho_{0}=\nu$, which is continuous on $t \geq 0$ (for the weak topology) and such that
\begin{equation}
\label{eqintegrflowcontinuity}
\int_{0}^{+\infty} \int_{M} \abs{A_{t}(x)} \, d\rho_{t}\, dt < +\infty.
\end{equation}
Then $(\rho_{t})_{t \geq 0}$ is a distributional solution to the continuity equation
	\begin{equation*}
	\left\{
		\begin{aligned}
		\partial_{t} \rho_{t} + \Div (\rho_{t} A_{t})&=0, \quad t >0\\
		\rho_{0} &= \nu
		\end{aligned}
	\right.
	\end{equation*}
if and only if $\rho_{t} = {S_{t}}_{\#} \nu$, where $(S_{t})_{t \geq 0}$ is the flow of diffeomorphisms associated to $A_{\bullet}$, in the sense of (\ref{eqodeflowat}).
\end{theorem}

%%%%% SECTION: Revisiting Bakry-\'Emery's $\Gamma$-calculus
\section{Revisiting Bakry-\'Emery's $\Gamma$-calculus}
\label{section3}
In this section we revisit Bakry-\'Emery's $\Gamma$-calculus, firstly recalling the basic definitions of $\Gamma$, its iteration $\Gamma_{2}$ and the curvature-dimension condition $\mathrm{CD}(\rho, \infty)$ together with its most essential consequences for our work. Secondly, we study further iterations of this theory, which give birth to the operators $\Gamma_{n}$. In particular, we will be interested in higher analogs $\mathrm{CD}(\rho, \infty)$, in the sense of $\Gamma_{n+1} \geq \rho_{n} \Gamma_{n}$, and their consequences for the associated semigroup $(\P_{t})_{t \geq 0}$.

We assume, unless we express the contrary, that $(M,g)$ is a complete and connected Riemannian manifold with weight $d\mu = e^{-W}\, d \mathrm{Vol}$ and we consider $\L = \Delta - \nabla W \cdot \nabla$.

%%%%% SUBSECTION: Classical $\Gamma$-calculus
\subsection{Classical $\Gamma$-calculus}

For each $f,h \in \C^{\infty}_{\mathrm{c}}(M)$, we define the usual carr\'e du champ operator $\Gamma$ as
\begin{equation*}
\label{eqdefgamma}
\Gamma(f,h) := \frac12\left(\L(fh) - f\,\L h - h\,\L f\right).
\end{equation*}
$\Gamma$ is a bilinear symmetric form which is moreover positive: for all $f \in \C^{\infty}_{\mathrm{c}}(M)$, $\Gamma(f,f) \geq 0$. We mostly work with evaluations of the form $\Gamma(f,f)$, so we define, doing an abuse of notation, $\Gamma(f) := \Gamma(f,f)$. In the case of a Markov diffusion operator $\L$ as (\ref{eqoperatorldef}), then $\Gamma$ has an explicit and natural representation in terms of Riemannian gradient:
\begin{equation}
\label{eqgammadiffusion}
\Gamma(f) = \abs{\nabla f}^{2}.
\end{equation}

Using (\ref{eqgammadiffusion}), we also have the Riemannian distance $d_{g}$ can be written in terms of the operator $\Gamma$. Indeed, for each $x,y \in M$,
\begin{equation*}
\label{eqriemanniandistancegamma}
d_{g}(x,y) = \sup_{\substack{f \in \C^{\infty}_{\mathrm{c}}(M); \\ \Gamma(f) \leq 1}} \left(f(x)- f(y)\right).
\end{equation*}
This directly leads to the following characterization of Lipschitz functions, which plays an important role in this note.

\begin{proposition}
\label{proplipschitzgamma}
Let $f \colon M \to \R$ be a smooth function. Then $f$ is $K$-Lipschitz for the distance $d_{g}$ if and only if $\sqrt{\Gamma(f)} \leq K$ uniformly on $M$.
\end{proposition}

In their seminal work \cite{MR0889476}, Bakry and \'Emery defined the iterated carr\'e du champ operator $\Gamma_{2}$: for $f,g \in \C^{\infty}_{\mathrm{c}}(M)$,
\begin{equation*}
\label{eqdefgamma2}
\Gamma_{2}(f,h) := \frac12\left(\L\Gamma(f,h) - \Gamma(f,\L h) - \Gamma(h, \L f)\right).
\end{equation*}
As we did so with $\Gamma$, we set $\Gamma_{2}(f):= \Gamma_{2}(f,f)$. Now if $\L = \Delta - \nabla W \cdot \nabla$, then using Bochner's formula (\ref{eqbochnerformula}) we can deduce an explicit formula for $\Gamma_{2}$:
\begin{equation}
\label{eqgamma2diffusion}
\Gamma_{2}(f) = \abs{\nabla^{2} f}_{\mathrm{HS}}^{2} + \Ric(\nabla f, \nabla f) + \nabla^{2} W(\nabla f, \nabla f).
\end{equation}
With both definitions of $\Gamma$ and $\Gamma_{2}$, we recall the classical $\mathrm{CD}(\rho, \infty)$ curvature-dimension condition, primevally defined by Bakry and \'Emery.

\begin{definition}[Curvature-dimension condition $\mathrm{CD}(\rho, \infty)$]
\label{defnbakryemerycd}
We say that the generator $\L$ satisfies the curvature-dimension condition $\mathrm{CD}(\rho, \infty)$ if there exists $\rho \in \R$ such that for each $f \in \C^{\infty}_{\mathrm{c}}(E)$,
$$\Gamma_{2}(f) \geq \rho \Gamma(f).$$
\end{definition}

\begin{remark}
\label{rkbakryemeryricci}
For the weighted manifold $(M, g, e^{-W} d\mathrm{Vol})$ we define the Bakry-\'Emery-Ricci tensor $\Ric_{W}:= \Ric + \nabla^{2} W$. Then we see that $\mathrm{CD}(\rho, \infty)$ holds if and only if $\Ric_{W}  \succeq \rho g$. More generally, the Bakry-\'Emery-Ricci tensor is an object of interest in itself in the context of weighted manifolds as it encodes the structure of the weight measure $e^{-W} d\mathrm{Vol}$ into a tensor. Assuming bounds on $\Ric_{W}$ implies as well topological and geometrical properties (see for example \cite{MR2016700}).
\end{remark}

The curvature-dimension condition $\mathrm{CD}(\rho, \infty)$ entails lots of properties for $\L$ and its associated semigroup $(\P_{t})_{t \geq 0}$. First of all, if we suppose that $(M,g)$ is complete and $\mathrm{CD}(\rho, \infty)$ holds, then its associated semigroup is conservative or non-explosive: for each $t \geq 0$, $\P_{t} \1 = \1$, where $\1$ denotes the constant function equal to 1 on $M$ (see for example Theorem 3.2.6 in \cite{bakry2013analysis}). Equivalently, if $\zeta$ denotes the explosion time of the process $(X_{t})$ generated by $(\P_{t})_{t \geq 0}$, then $\mathbb{P}(\zeta = +\infty)=1$.

Under $\mathrm{CD}(\rho, \infty)$, the invariant measure $\mu$ satisfies at the same time the Poincar\'e and logarithmic Sobolev inequalities (see for example Propositions 4.8.1 and 5.7.1 in \cite{bakry2013analysis}).
\begin{theorem}
\label{thmpoincareandlogsobolev}
Suppose $\mathrm{CD}(\rho, \infty)$ holds for $\rho >0$. Then:
	\begin{enumerate}[(a)]
	\item $(M, g, \mu)$ satisfies a Poincar\'e inequality with constant $\frac1\rho$: that is, for each $f \in \C^{\infty}_{\mathrm{c}}(M)$,
	$$\Var_{\mu}(f) \leq \frac1\rho \int_{M} \Gamma(f)\, d\mu.$$
	\item $(M, g, \mu)$ satisfies a log-Sobolev inequality with constant $\frac1\rho$: that is, for each $f \in \C^{\infty}_{\mathrm{c}}(M)$,
	$$\Ent_{\mu}(f^{2}) \leq \frac2\rho \int_{M} \Gamma(f)\, d\mu,$$
	where $\Ent_{\mu}(\cdot)$ denotes the entropy with respect to the measure $\mu$: for a non-negative measurable function $h\colon M \to \R$, $\Ent_{\mu}(h):= \int_{M} h \log h \, d\mu - \left(\int_{M} h \, d\mu\right) \log \left(\int_{M} h \, d\mu\right).$
	\end{enumerate}
\end{theorem}

Another remarkable consequence of $\mathrm{CD}(\rho, \infty)$ is reflected analytically on the semigroup, via the local Poincar\'e and logarithmic Sobolev inequalities (for a proof, see for example Theorems 4.7.2 and 5.5.2 in \cite{bakry2013analysis}).
\begin{theorem}
\label{thmlocalinequalitiescurvature}
The following assertions are equivalent.
	\begin{enumerate}[(i)]
	\item $\mathrm{CD}(\rho, \infty)$ holds.
	\item For each $f \in \C^{\infty}_{\mathrm{c}}(M)$ and every $t \geq 0$,
	$$\Gamma(\P_{t}f) \leq e^{-2\rho t} \P_{t}(\Gamma(f)).$$
	\item For each $f \in \C^{\infty}_{\mathrm{c}}(M)$ and every $t \geq 0$,
	$$\sqrt{\Gamma(\P_{t}f)} \leq e^{-\rho t} \P_{t}(\sqrt{\Gamma(f)}).$$
	\item For each $f \in \C^{\infty}_{\mathrm{c}}(M)$ and every $t \geq 0$, if $\rho \neq 0$,
	$$\P_{t}(f^{2})-(\P_{t}f)^{2} \leq \frac{1-e^{-2\rho t}}{\rho}\P_{t}(\Gamma(f));$$
	if $\rho =0$, then
	$$\P_{t}(f^{2})-(\P_{t}f)^{2} \leq 2t \P_{t}(\Gamma(f)).$$
	\item For each $f \in \C^{\infty}_{\mathrm{c}}(M)$ positive and every $t \geq 0$, if $\rho \neq 0$,
	$$\P_{t}(f^{2} \log f^{2})-\P_{t} (f^{2}) \log \P_{t}(f^{2}) \leq 2 \frac{1-e^{-2\rho t}}{\rho}\P_{t}\left(\Gamma(f)\right);$$
	if $\rho =0$, then
	$$\P_{t}(f^{2} \log f^{2})-\P_{t} (f^{2}) \log \P_{t}(f^{2}) \leq 4t\P_{t}\left(\Gamma(f)\right).$$
	\end{enumerate}
\end{theorem}

%%%%% RK SOBRE L AUTOADJUNTO %%%%%
\begin{remark}
In this note we aim to develop and use properties arising from $\Gamma$-calculus theory, which is capable to give insights on the analytical or probabilistic object represented by the generator $\L$ in terms of more manipulable objects (algebraically speaking) such as the $\Gamma_{n}$ operators (to be defined in the next subsection). To do so, a class of functions $\mathcal{A}_{0}$ has to be fixed, in order to define those functionals and work properly with them. In fact, this was already done (implicitly) in the very first paragraphs of this subsection, by taking $\mathcal{A}_{0} = \C^{\infty}_{\mathrm{c}}(M)$ as $\L$ is essentially self-adjoint on this class of functions (c.f. \cite[Chapter 3]{bakry2013analysis}).
\end{remark}

%%%%% SUBSECTION: Further iterations of $\Gamma$-calculus
\subsection{Further iterations of $\Gamma$-calculus}

As it was seen before, the construction of $\Gamma_{2}$ was iterated from $\Gamma$, so it is possible to go further and construct $\Gamma_{n}$, for $n \geq 3$; this idea was already explored before in \cite{zbMATH00559130}, \cite{zbMATH00563664} and \cite{MR3320893}. If we set $\Gamma_{0}(f,h):= fh$, then for $n\geq 0$ and $f,h \in \C^{\infty}_{\mathrm{c}}(M)$ we define
\begin{equation}
\label{eqdefgamman}
\Gamma_{n+1}(f,h) := \frac12\left(\L\Gamma_{n}(f,h) - \Gamma_{n}(f,\L h) - \Gamma_{n}(h, \L f)\right),
\end{equation}
and again we set $\Gamma_{n}(f):= \Gamma_{n}(f,f)$.

\begin{remark}
Iterations in the sense of (\ref{eqdefgamman}) are consistent with classical $\Gamma$-calculus, as $\Gamma_{1} = \Gamma$. For the sequel, we will always assume that the iterations start from $\Gamma_{0}$: that is, $\Gamma_{1}$ is the first iteration, $\Gamma_{2}$ the second and so on. For the purposes of this note we are concerned only up to the third iteration, that is, $\Gamma_{3}$, but the results in this section are stated and proved for general iterations.
\end{remark}

The goal of this section is to obtain analytical bounds for the semigroup $(\P_{t})_{t \geq 0}$ assuming that $\Gamma_{n+1}(f) \geq \rho_{n} \Gamma_{n}(f)$, in the spirit of Theorem \ref{thmlocalinequalitiescurvature}. Now we provide some illustrative examples of diffusion operators satisfying these criteria.

%%%%% EJEMPLOS %%%%%
\begin{example}[Laplace-Beltrami operator on $\mathbb{S}^{d}$]
\label{explesphere}
Let $d \geq 2$ and consider the $d$-dimensional sphere $\mathbb{S}^{d} := \{x \in \R^{d+1}: \abs{x}=1\}$ together with its standard Riemannian metric $g$ (induced by its embedding on $\R^{d+1}$), which has a constant Ricci tensor: $\Ric = (d-1) g$. Let us fix $\mu = \mathrm{Vol}$, the uniform measure on $\mathbb{S}^{d}$, and let $\L = \Delta$, the Laplace-Beltrami operator on $\mathbb{S}^{d}$, which is essentially self-adjoint on $\C^{\infty}_{\mathrm{c}}(M) = \mathcal{C}^{\infty}(M)$ as $\mathbb{S}^{d}$ is complete.

The stochastic process generated by $\L$ is the (rescaled) Brownian motion on $\mathbb{S}^{d}$: $X_{t} = \sqrt2 B_{t}$, for $t \geq 0$, where $(B_{t})_{t \geq 0}$ the Brownian motion on the sphere.  Following the previous discusion, we have immediately that for $f \colon \mathbb{S}^{d}\to \R$ smooth,
$$\Gamma(f) = \abs{\nabla f}^{2}.$$
For $\Gamma_{2}$, we can stress the fact that Ricci tensor is constant, so by (\ref{eqgamma2diffusion}), we get
\begin{align*}
\Gamma_{2}(f)&= \abs{\nabla^{2} f}_{\mathrm{HS}}^{2} + (d-1) \abs{\nabla f}^{2}
\end{align*}
and by polarization we get
\begin{align*}
\Gamma_{2}(f,h)&= \langle \nabla^{2} f, \nabla^{2} h \rangle_{\mathrm{HS}} + (d-1) \nabla f\cdot \nabla h.
\end{align*}
We can observe that $\Gamma_{2}(f) \geq (d-1) \Gamma(f)$, so $\L$ satisfies $\mathrm{CD}(d-1, \infty)$. In particular, we have that the process $(X_{t})_{t \geq 0}$ is non-explosive and thus neither the Brownian motion on the sphere.

Now let us compute $\Gamma_{3}$:
\begin{align*}
\Gamma_{3}(f) &= \frac12 \Delta \Gamma_{2}(f) - \Gamma_{2}(f, \Delta f)\\
&=\left(\frac12 \Delta \abs{\nabla^{2} f}_{\mathrm{HS}}^{2} + \frac12 (d-1) \Delta \abs{\nabla f}^{2}\right) - \left(\langle \nabla^{2} f, \nabla^{2} (\Delta f) \rangle_{\mathrm{HS}} + (d-1)  \nabla f\cdot \nabla(\Delta f)\right).
\end{align*}
Using Bochner's formula (\ref{eqbochnerformula}) and again the constant Ricci curvature, we obtain
$$\frac12 (d-1) \Delta \abs{\nabla f}^{2} = (d-1) \abs{\nabla^{2} f}_{\mathrm{HS}}^{2} + (d-1) \nabla f \cdot \nabla(\Delta f) + (d-1)^{2}\abs{\nabla f}^{2},$$
\end{example}
so
\begin{align*}
\Gamma_{3}(f) &= \frac12 \Delta \abs{\nabla^{2} f}_{\mathrm{HS}}^{2} - \langle \nabla^{2} f, \nabla^{2} (\Delta f) \rangle_{\mathrm{HS}} + (d-1) \abs{\nabla^{2} f}_{\mathrm{HS}}^{2} + (d-1)^{2}\abs{\nabla f}^{2}\\
&=\frac12 \Delta \abs{\nabla^{2} f}_{\mathrm{HS}}^{2} - \langle \nabla^{2} f, \nabla^{2} (\Delta f) \rangle_{\mathrm{HS}} + (d-1) \Gamma_{2}(f).
\end{align*}
Now if we choose a local system of coordinates, then we are able to use again Bochner's formula on each coordinate of $\nabla f$:
\begin{align*}
\frac12 \Delta \abs{\nabla \nabla^{i} f}^{2} &= \abs{\nabla^{2} \nabla^{i} f}_{\mathrm{HS}}^{2} + \nabla \nabla^{i} f \cdot \nabla(\Delta (\nabla^{i} f)) + \Ric(\nabla \nabla^{i} f, \nabla \nabla^{i}f)\\
&= \abs{\nabla^{2} \nabla^{i} f}_{\mathrm{HS}}^{2} + \nabla \nabla^{i} f \cdot \nabla \nabla^{i} (\Delta f) + \Ric(\nabla \nabla^{i} f, \nabla \nabla^{i}f)\\
&= \abs{\nabla^{2} \nabla^{i} f}_{\mathrm{HS}}^{2} + \nabla \nabla^{i} f \cdot \nabla \nabla^{i} (\Delta f) + (d-1)g(\nabla \nabla^{i} f, \nabla \nabla^{i}f)\\
&\geq \abs{\nabla^{2} \nabla^{i} f}_{\mathrm{HS}}^{2} + \nabla \nabla^{i} f \cdot \nabla \nabla^{i} (\Delta f),
\end{align*}
so if we sum in $i$, we get
$$\frac12 \Delta \abs{\nabla^{2} f}_{\mathrm{HS}}^{2} \geq \abs{\nabla^{3} f}_{\mathrm{HS}}^{2} + \langle \nabla^{2} f, \nabla^{2}(\Delta f) \rangle,$$
so
$$\Gamma_{3}(f) \geq \abs{\nabla^{3} f}_{\mathrm{HS}}^{2} + (d-1) \Gamma_{2}(f);$$
thus we conclude that $\Gamma_{3}(f) \geq (d-1) \Gamma_{2}(f)$.

\begin{example}[Laguerre generator on $(0,+\infty)^{d}$]
\label{explelaguerre}
Fix $p \in (0,+\infty)^{d}$ and define on $(0,+\infty)^{d}$ the diffusion operator
$$\L_{p} f := \sum_{i=1}^{d} x_{i} \frac{\partial^{2} f}{\partial x_{i}^{2}} + \sum_{i=1}^{d}(p_{i}-x_{i}) \frac{\partial f}{\partial x_{i}},$$
acting on smooth functions $f\colon (0,+\infty)^{d} \to \R$ and let $\mu$ be the multivariate gamma law on $(0,+\infty)^{d}$; i.e., $d\mu = \prod_{i=1}^{d}\frac1{\gamma(p_{i})} x_{i}^{p_{i}-1}e^{-x_{i}}\, dx_{i}$ (here $\gamma$ denotes the usual gamma function). Following Remark \ref{rkmuchotexto}, here the function $x \mapsto (\frac1{x_{i}} \delta_{ij})$ defines a Riemannian metric $g$ on $(0,+\infty)^{d}$.

Its associated diffusion process $(X_{t})_{t \geq 0}$ is non-explosive as soon as for each $1 \leq i \leq d$, $p_{i} \geq 1$ \cite{MR1997032} and on the other hand, it is known (see for example \cite{MR0649183}) that if for each $1 \leq i \leq d$, $p_{i} \geq \frac32$, then $\L_{p}$ is essentially self-adjoint on $\C^{\infty}_{\mathrm{c}}((0,+\infty)^{d})$, so we will restrict our attention towards $p_{i} \geq \frac32$.

$\Gamma_{n}$ operators have an explicit representation on terms of the derivatives of $f$ (see for example \cite{MR3320893}):
\begin{align*}
\Gamma (f) &= \sum_{i=1}^d x_i \left(\frac{\partial f}{\partial x_i}\right)^2 \\
\Gamma_2(f) &= \sum_{i,j=1}^d x_i x_j \left(\frac{\partial^2 f}{\partial x_i\partial x_j} \right)^2
    + \sum_{i=1}^d x_i \, \frac{\partial f}{\partial x_i} \,  \frac{\partial^2 f}{\partial x_i^2}
    + \frac {1}{2} \sum_{i=1}^d (p_i+x_i) \left( \frac{\partial f}{\partial x_i}\right)^2 \\
 \Gamma _3(f) & = \sum_{i,j, k=1}^d x_i x_j x_k \left(\frac{\partial^3 f}{\partial x_i\partial x_j\partial x_k} \right)^2
     + 3 \sum_{i,j=1}^d x_i x_j  \, \frac{\partial^2 f}{\partial x_i\partial x_j} \, \frac{\partial^3 f}{\partial x_i^2\partial x_j} \\
     &\,\,\,\,\,\,\,\,\,  + \frac {3}{2} \sum_{i,j=1}^d (p_i+x_i) x_j\left(\frac{\partial^2 f}{\partial x_i\partial x_j} \right)^2
       +  \frac {3}{2}  \sum_{i=1}^d x_i \left(\frac{\partial^2 f}{\partial x_i^2}\right)^2   \\
&\,\,\,\,\,\,\,\,\,    + \frac {3}{2} \sum_{i=1}^d x_i \, \frac{\partial f}{\partial x_i} \,  \frac{\partial^2 f}{\partial x_i^2}
       +  \frac {1}{4}  \sum_{i=1}^d (3p_i +x_i) \left(\frac{\partial f}{\partial x_i}\right)^2.
\end{align*}

As $p_{i} \geq \frac32$, we can easily observe that $\Gamma_{2}(f) \geq \frac12 \Gamma(f)$ and that $\Gamma_{3}(f) \geq \frac12 \Gamma_{2}(f)$.

	\begin{remark}
	\label{rkexpdistrlaguerre}
	In the special case when $d =1$ and $p=1$, then $\mu$ corresponds to the exponential measure on $(0,+\infty)$. $\L_{p}$ is not essentially self-adjoint on $\C^{\infty}_{\mathrm{c}}((0,+\infty))$ but it is on the space $\C^{\infty}_{\mathrm{c,\, Neu}}((0,+\infty))$, defined by
	$$\C^{\infty}_{\mathrm{c,\, Neu}}((0,+\infty)) := \left\{f \in \C^{\infty}_{\mathrm{c}}((0,+\infty)) : \lim_{x \to 0} x^{p} e^{-x} f'(x) =0 \right\} \subseteq \C^{\infty}_{\mathrm{c}}((0,+\infty)),$$
	which could be interpreted as imposing a Neumann boundary condition at 0 on $(0, +\infty)$. On the other hand, it can be seen that in this case, it holds that $\Gamma_{2}(f) \geq \frac12 \Gamma(f)$ and that $\Gamma_{3}(f) \geq \frac12 \Gamma_{2}(f)$
	\end{remark}
\end{example}

Now we show a technical lemma which will be the basis of our calculations. It corresponds to a classical result when $n=1$ \cite{bakry2013analysis}, which is used to obtain Theorem \ref{thmlocalinequalitiescurvature}, but for the sake of completeness we provide a full proof.

\begin{lem}
\label{lemma1}
Let $n \geq 0$, $f \in \C^{\infty}_{\mathrm{c}}(M)$, $t \geq 0$ and $x \in M$. We define $\Lambda_{n}\colon \R_{+} \to \R_{+}$ as
$$\Lambda_{n}(s) := \P_{s}(\Gamma_{n}(\P_{t-s}f))(x).$$
Then for each $t \geq s$,
	\begin{equation*}
	\frac{d}{ds} \Lambda_{n}(s) = 2 \P_{s}(\Gamma_{n+1}(\P_{t-s}f)) = 2\Lambda_{n+1}(s).
	\end{equation*}
\end{lem}

\begin{proof}
Let us define
	\begin{align*}
	\phi_{1}&\colon [0, t] \to \C^{\infty}(M) \cap L^{\infty}(\mu),&  &s \mapsto \phi_{1}(s) := \P_{t-s}f;&\\
	\phi_{2}&\colon \C^{\infty}(M)\cap L^{\infty}(\mu) \to \C^{\infty}(M)\cap L^{\infty}(\mu),&  &u \mapsto \phi_{2}(u) := \Gamma_{n}(u);&\\
	\phi_{3}&\colon \R_{+} \times \C^{\infty}(M)\cap L^{\infty}(\mu) \to \C^{\infty}(M)\cap L^{\infty}(\mu),&  &(s,u) \mapsto \phi_{3}(s,u) := \P_{s}u.&\
	\end{align*}
We can note that $\Lambda_{n}(s) = \phi_{3}(s, \phi_{2}(\phi_{1}(s)))$. Using the chain rule we obtain
	\begin{align}
	\label{eqdergamman}
	\frac{d}{ds}\Lambda_{n}(s) = D_{s} \phi_{3}(s, \Gamma_{n}(\P_{t-s}f)) + D_{u} \phi_{3}(s, \Gamma_{n}(\P_{t-s}f)) D\phi_{2}(\P_{t-s}f)(\phi_{1}'(s)).
	\end{align}
On the other hand, we observe that
	\begin{align*}
	D_{s} \phi_{3}(s,u) &= \lim_{\epsilon \to 0} \frac{\phi_{3}(s+\epsilon, u) - \phi_{3}(s,u)}{\epsilon} = \lim_{\epsilon \to 0} \frac{\P_{s + \epsilon} u - \P_{s} u}{\epsilon} = \lim_{\epsilon \to 0} \P_{s} \left(\frac{\P_{\epsilon} u - \P_{0}u}{\epsilon}\right)\\
	&=\P_{s} \left(\lim_{\epsilon \to 0} \frac{\P_{\epsilon} u - \P_{0}u}{\epsilon}\right)\\
	&= \P_{s} \L u.
	\end{align*}
For the second variable, we observe that $\phi_{3}$ is linear in its second variable, so
$$D_{u} \phi_{3}(s,u)(h) = \phi_{3}(s,h).$$
We know that $\Gamma_{n}(f) = \Gamma_{n}(f,f)$, with $\Gamma_{n}(\cdot, \cdot)$ a symmetric bilinear form, so
$$D\phi_{2}(u)(h) = 2 \Gamma_{n} (u,h).$$
For $\frac{d}{ds} \varphi_{1}$, we note that
	\begin{align*}
	\frac{d}{ds}\phi_{1}(s) = \frac{d}{ds}\left(\phi_{3}(t-s, f)\right)= - \partial_{s}\phi_{3}(t-s, f) = - \P_{t-s} \L f.
	\end{align*}
Finally, mixing all this into (\ref{eqdergamman}), we get
	\begin{align*}
	\frac{d}{ds}\Lambda_{n}(s) &= \P_{s}(\L(\Gamma_{n}(\P_{t-s}f))) - 2\P_{s}(\Gamma_{n}(\P_{t-s}f, \P_{t-s}\L f)) =\P_{s}(\Gamma_{n+1}(\P_{t-s}f)).\
	\end{align*}
\end{proof}

Now the following proposition is just a generalization of one of the consequences of the curvature condition $\mathrm{CD}(\rho_{1}, \infty)$, reflected in Theorem \ref{thmlocalinequalitiescurvature}, in the context of higher iterations of $\Gamma$-calculus.

\begin{proposition}
\label{lemma3}
Let $n \geq 0$ and suppose that there exists a constant $\rho_{n} \in \R$ such that for all $f \in  \C^{\infty}_{\mathrm{c}}(M)$, $\Gamma_{n+1}(f) \geq \rho_{n} \Gamma_{n}(f)$. Then for each $t \geq 0$,
	\begin{equation*}
	\Gamma_{n}(\P_{t}f) \leq e^{-2\rho_{n} t} \P_{t}(\Gamma_{n}(f)).
	\end{equation*}
\end{proposition}

\begin{proof}
Applying Lemma \ref{lemma1} and imposing $\Gamma_{n+1} \geq \rho_{n} \Gamma_{n}$, we get
	\begin{align*}
	\frac{d}{ds} \Lambda_{n}(s) = 2 \P_{s}(\Gamma_{n+1}(\P_{t-s}f)) \geq 2 \rho_{n} \P_{s}(\Gamma_{n}(\P_{t-s}f)) = \Lambda_{n}(s).
	\end{align*}
Then by Gr\"onwall's inequality, we obtain
$$\Lambda_{n}(0) \leq e^{-2\rho_{n} t} \Lambda_{n}(t),$$
so
$$\Gamma_{n}(\P_{t}f) \leq e^{-2\rho_{n} t} \P_{t}(\Gamma_{n}(f)).$$
\end{proof}

Now we present another bound.

\begin{proposition}
\label{lemma2}
Let $n \geq 1$ and suppose that for any $h\in  \C^{\infty}_{\mathrm{c}}(M)$, both $\Gamma_{n-1}(h)$ and $\Gamma_{n+1}(h)$ are non-negative. Then for each $f\in  \C^{\infty}_{\mathrm{c}}(M)$ and for all $t > 0$, we have that
	\begin{equation*}
	\Gamma_{n}(\mathrm{P}_{t} f) \leq \frac{1}{2t} \mathrm{P}_{t}(\Gamma_{n-1}(f)).
	\end{equation*}
\end{proposition}

\begin{proof}
Using Lemma \ref{lemma1}, we get
$$\Lambda_{n}'(s) = 2 \P_{s}(\Gamma_{n+1}(\P_{t+s}f)) \geq 0,$$
so $\Lambda_{n}$ is non-decreasing. Then, for each $t \geq s \geq 0$, we infer that
$$\Gamma_{n}(\P_{t}f) = \P_{0}(\Gamma_{n}(\P_{t}f)) = \Lambda_{n}(0) \leq \Lambda_{n}(s) = \P_{s}(\Gamma_{n}(\P_{t-s}f)).$$
Thus integrating and applying again Lemma \ref{lemma1}, we see that
	\begin{align*}
	t \Gamma_{n}(\P_{t}f) &= \int_{0}^{t} \Gamma_{n}(\P_{t}f) \, ds \leq \int_{0}^{t} \P_{s}(\Gamma_{n}(\P_{t-s})f) \, ds = \frac12 \int_{0}^{t} \Lambda_{n-1}'(s) \, ds\\
	& = \frac12 \left(\Lambda_{n-1}(t) - \Lambda_{n-1}(0)\right)\\
	&= \frac12 \left(\P_{t}(\Gamma_{n-1}(f)) - \Gamma_{n-1}(\P_{t}f)\right),
	\end{align*}
from where we get
$$\Gamma_{n}(\P_{t}f) \leq \frac{1}{2t} \left(\P_{t}(\Gamma_{n-1}(f)) - \Gamma_{n-1}(\P_{t}f)\right) \leq \frac{1}{2t} \P_{t}(\Gamma_{n-1}(f)).$$
\end{proof}

By mixing the two previous propositions, we deduce the following inequality.

\begin{proposition}
\label{lemma4}
Let $n \geq 1$ and suppose that for any $h\in  \C^{\infty}_{\mathrm{c}}(M)$, both $\Gamma_{n-1}(h)$ and $\Gamma_{n+1}(h)$ are non-negative. Additionally, let us assume that there exists a constant $\rho_{n} \in \R$ such that for each $f \in \C^{\infty}_{\mathrm{c}}(M)$, $\Gamma_{n+1}(f) \geq \rho_{n} \Gamma_{n}(f)$. Then for all $t > 0$,
	\begin{equation*}
	\Gamma_{n}(\P_{t}f) \leq \frac1t e^{-\rho_{n} t} \P_{t} (\Gamma_{n-1} f).
	\end{equation*}
\end{proposition}

\begin{proof}
Let $t > 0$. By Proposition \ref{lemma3}:
	\begin{equation}
	\label{lemma4eq1}
	\Gamma_{n}(\P_{t}f) = \Gamma_{n}(\P_{t/2}(\P_{t/2}f)) \leq e^{-\rho_{n} t} \P_{t/2}(\Gamma_{n}(\P_{t/2}f)).
	\end{equation}
Then, applying Proposition \ref{lemma2}, we know that
	\begin{equation}
	\label{lemma4eq2}
	\Gamma_{n}(\P_{t/2}f) \leq \frac1t \P_{t/2}(\Gamma_{n-1}(f)).
	\end{equation}
Therefore, mixing up inequalities (\ref{lemma4eq1}) and (\ref{lemma4eq2}), we finally obtain
$$\Gamma_{n}(\P_{t}f) \leq \frac1t e^{-\rho_{n} t} \P_{t} (\Gamma_{n-1} f).$$
\end{proof}

\begin{remark}
\label{rknonnegative}
Let us note that if $n=1$ and if there exist some constants $\rho_{1}, \rho_{2} > 0$ such that for all $f \in \C^{\infty}_{\mathrm{c}}(M)$, $\Gamma_{3}(f) \geq \rho_{2} \Gamma_{2}(f)$ and $\Gamma_{2}(f) \geq \rho_{1} \Gamma_{1}(f)$, then the non-negativeness hypotheses in Proposition \ref{lemma4} are then satisfied.
\end{remark}

%%%%% SECTION: The diffusion transport map
\section{The diffusion transport map on smooth manifolds}
\label{section4}
In this section we prove the main result of this note, Theorem \ref{thmmainresult}, which employs Kim and Milman's heat flow transport map \cite{kimmilman2012} on a weighted manifold setting to obtain a Lipschitz transport map between the prescripted measure on the manifold and a log-Lipschitz perturbation of the former.

Firstly, the construction will be detailed, taking care with the proper details that must be kept with attention when passing from the Euclidean setting towards the manifold one. After that, using the machinery of $\Gamma$-calculus exhibited in Section \ref{section3}, we prove that the heat flow transport map is Lipschitz, giving an explicit bound on its Lipschitz constant.

Here the context will always be given by a complete and connected Riemannian manifold $(M,g)$ with weight $d\mu = e^{-W}\, d \mathrm{Vol}$ and the natural diffusion operator $\L$ given by $\L = \Delta - \nabla W \cdot \nabla$, unless otherwise stated. On the other hand, we assume that there exist constants $\rho_{1}, \rho_{2} > 0$ such that for any $f \in \C^{\infty}_{\mathrm{c}}(M)$, $\Gamma_{2}(f) \geq \rho_{1} \Gamma(f)$ (i.e., $\mathrm{CD}(\rho_{1}, \infty)$) and $\Gamma_{3}(f) \geq \rho_{2} \Gamma_{2}(f)$.
	
We emphasize that under these assumptions, $\L$ is essentially self-adjoint on $\C^{\infty}_{\mathrm{c}}(M)$ and its associated semigroup is conservative (or equivalently, that the associated Markov process does not explode).

\subsection{Construction of the transport map}
Let $V\colon M \to \R$ be a smooth function which is $K$-Lipschitz for the metric $d_{g}$; so, in the light of Proposition \ref{proplipschitzgamma}, $V$ is such that $\sqrt{\Gamma_{1}(V)} \leq K$ uniformly on $M$. Let us define $f=e^{-V}$. The following result justifies the integrability of $f$ with respect to $\mu$ under $\mathrm{CD}(\rho, \infty)$, with $\rho >0$; it corresponds to Herbst's argument\footnote{This unpublished result was developed by Ira Herbst in a letter addressed to Leonard Gross: \url{https://perso.math.univ-toulouse.fr/ledoux/files/2018/10/Herbst.pdf}} (see for example Proposition 5.4.1 in \cite{bakry2013analysis}).

\begin{lem}
\label{lemmalipisintegrable}
Let $V\colon M \to \R$ be a smooth function such that $\sqrt{\Gamma_{1}(V)} \leq K$. If there exists $\rho > 0$ such that $\mathrm{CD}(\rho, \infty)$ holds, then both $V$ and $e^{sV}$ are $\mu$-integrable, for any $s \in \R$. In particular, $e^{-V}$ is in $L^{p}(\mu)$, for each $p \geq 1$.
\end{lem}

Without loss of generality we will assume $\int_{M} f \, d\mu = \int_{M} e^{-V} \, d\mu = 1$ so we can define the probability measure $\nu$ on $M$ as $d\nu := f \, d\mu$. For each $t \geq 0$, set $d\rho_{t} = \P_{t} f\, d\mu$, which is a probability measure, as $\mu$ is invariant for $(\P_{t})_{t \geq 0}$. In particular, we have that the flow $(\rho_{t})_{t \geq 0}$ is such that $\rho_{0} = \nu$ and $\rho_{\infty}:=\lim_{t \to +\infty} \rho_{t} = \mu$ in distribution. Indeed, $f \in \L^{2}(\mu)$ (because of Lemma \ref{lemmalipisintegrable}) and we have that $\L$ is ergodic and its semigroup is conservative, then $\P_{t}f \to 1$ in $L^{2}(\mu)$, which yields the desired convergence.

Now let us define, for each $t \geq 0$ and $x \in M$, $V_{t}(x):= - \log \P_{t} f(x)$, so that $\nabla V_{t}$ is a vector field for each $t \geq 0$. The following proposition shows that actually that $V_{\bullet}$ and $(\rho_{t})_{t \geq 0}$ are connected via the continuity equation.

\begin{proposition}
\label{propcontinuityequation}
The flow $(\rho_{t})_{t \geq 0}$ satisfies the continuity equation
	\begin{equation*}
	\left\{
		\begin{aligned}
		\partial_{t} \rho_{t} + \Div (\rho_{t} \nabla V_{t})&=0, \quad t >0\\
		\rho_{0} &= \nu
		\end{aligned}
	\right.
	\end{equation*}
in the distributional sense.
\end{proposition}

\begin{proof}
Let $t >0$ and $\varphi \in \C^{\infty}_{\mathrm{c}}(M)$. We shall prove that
$$\frac{d}{dt} \int_{M} \phi \, d\rho_{t} = \int_{M} \nabla \varphi \cdot \nabla V_{t}  \, d\rho_{t}.$$
Indeed, as $(\rho_{t})_{t \geq 0}$ is solution of the SDE (\ref{eqsdediffusion}), then it satisfies Fokker-Planck's equation (\ref{eqfokkerplanck}) in the distributional sense, hence we can write
	\begin{align*}
	\frac{d}{dt} \int_{M} \phi \, d\rho_{t} &= \int_{M} \left(\Delta \phi - \nabla W \cdot \nabla \phi \right)\, d\rho_{t} = \int_{M} \Delta \phi e^{-W} \P_{t}f \, d \mathrm{Vol} - \int_{M} \nabla W \cdot \nabla \phi \, d \rho_{t}\\
	&= - \int_{M} \nabla \phi \cdot \nabla (e^{-W} \P_{t}f) \, d \mathrm{Vol} -  \int_{M} \nabla W \cdot \nabla \phi \, d \rho_{t}\\
	&= - \int_{M} \nabla \phi \cdot \nabla \P_{t} f\, d\mu = -\int_{M} \nabla \phi \cdot \frac{\nabla \P_{t} f}{\P_{t} f} \, d \rho_{t}= \int_{M} \nabla \phi \cdot \nabla V_{t} \, d \rho_{t},
	\end{align*}
which was the desired equality.
\end{proof}

Now let us consider the flow of diffeomorphisms $(S_{t})_{t \geq 0}$ induced by the vector field $\nabla V_{t}$:
\begin{equation}
\label{eqodeflowvt}
\left\{
		\begin{aligned}
		\frac{d}{dt}S_{t}(x) &= \nabla V_{t}(S_{t}(x)), \quad t >0\\
		S_{0}(x)&=x.
		\end{aligned}
	\right.
\end{equation}

Our goal will be to characterize the flow $(\rho_{t})_{t \geq 0}$ in terms of $(S_{t})_{t \geq 0}$ using Theorem \ref{thmcontinuityequation}. First of all, it is clear that $\nabla V_{\bullet}$ is smooth (thus locally Lipschitz) as the semigroup $(\P_{t})_{t\geq 0}$ preserves smoothness as $\L$ is an elliptic diffusion operator. It is also clear that $t \mapsto \rho_{t}$ is continuous for the weak topology. Finally, we just have to justify the integrability condition (\ref{eqintegrflowcontinuity}), which will hold under the curvature-dimension condition $\mathrm{CD}(\rho_{1}, \infty)$.

\begin{lem}
\label{lemintegrability}
Let $V_{\bullet}$ be defined as above. If there exists $\rho_{1} >0$ such that $\mathrm{CD}(\rho_{1}, \infty)$ holds, then
$$\int_{0}^{+\infty} \int_{M} \abs{\nabla V_{t}(x)} \, d \rho_{t} \, dt< +\infty.$$
\end{lem}

\begin{proof}
Under $\mathrm{CD}(\rho_{1}, \infty)$ we may use the commutation between $\Gamma$ and $\P_{t}$ (item (iii) in Theorem \ref{thmlocalinequalitiescurvature}), so we have that for each $t \geq 0$,
$$\sqrt{\Gamma(\P_{t}f)} \leq e^{-\rho_{1}t} \P_{t}(\sqrt{\Gamma(f)}).$$
If we develop $\P_{t}(\sqrt{\Gamma(f)})$, we obtain
$$\P_{t}(\sqrt{\Gamma(f)}) = \P_{t}(\abs{\nabla f}) = \P_{t}(\abs{\nabla V} f).$$
As $V$ is a $K$-Lipschitz potential, then by definition we have that $\abs{\nabla V} \leq K$; thus we finally get
$$\abs{\nabla \P_{t}f} = \sqrt{\Gamma(\P_{t}f)} \leq e^{-\rho_{1}t} K \P_{t}f,$$
so it follows that $\abs{\nabla V_{t}(x)} \leq K e^{-\rho_{1}t}$. Then
$$\int_{0}^{+\infty} \int_{M} \abs{\nabla V_{t}(x)} \, d \rho_{t} \, dt \leq K \int_{0}^{+\infty} e^{-\rho_{1}t} \, dt <+\infty,$$
as $\rho_{1} > 0$.
\end{proof}

As we reunited the hypotheses of Theorem \ref{thmcontinuityequation}, we therefore deduce that for each $t \geq 0$, ${S_{t}}_{\#} \nu = \rho_{t}$. For each $t \geq 0$, let $T_{t} := S_{t}^{-1}$, so ${T_{t}}_{\#} \rho_{t} = \nu$.

If we suppose that for each $t \geq 0$ there exists a constant $K_{t} > 0$ such that $T_{t}$ is $K_{t}$-Lipschitz, and that $K:= \limsup_{t \to +\infty} K_{t} < +\infty$, then Lemma 1 in \cite{mikulincer2022lipschitz} states that (modulo subsequence) for each $x \in M$, the limit $\lim_{t \to +\infty} T_{t}(x)$ exists and moreover, it defines a $K$-Lipschitz mapping $T\colon M \to M$ such that $T_{\#}\mu = \nu$, as a consequence of Arzel\`a-Ascoli's theorem.

Now the following lemma exhibits a condition on $V_{t}$ to ensure that the mapping $T$ is Lipschitz, giving an estimate on its Lipschitz constant.

\begin{lem}
\label{lemboundonvtforlipschitz}
If there exists $\lambda\colon \R_{+} \to \R_{+}$ integrable such that for all $t \geq 0$ and $x\in M$,
$$-\nabla^{2} V_{t}(x)\preceq \lambda(t) g_{x},$$
then $T$ is $\exp\left(\int_{0}^{+\infty} \lambda(t)\, dt\right)$-Lipschitz.
\end{lem}

We give just a formal proof (see for example Proposition 1 in \cite{fathi2023transportation} and Lemma 2.1 in \cite{neeman2022lipschitz} for a true proof): if we look at the flow (\ref{eqodeflowvt}), then for any $x, y \in M$ we may define $\alpha(t):= d_{g}^{2}(S_{t}(x),S_{t}(y))$. If we derive it with respect to $t$, let us recall that $\frac{d}{dt}S_{t}(x) = \nabla V_{t}(S_{t}(x))$ and $\frac{d}{dt}S_{t}(y) = \nabla V_{t}(S_{t}(y))$, so the uniform bound $-\nabla^{2} V_{t}(x)\preceq \lambda(t) g_{x}$ allows us to bound it by below:
$$\frac{d}{dt} \alpha(t) \geq - 2 \lambda(t) \alpha(t).$$
Therefore Gr\"onwall's inequality implies that
$$d_{g}(S_{t}(x), S_{t}(y)) \geq d_{g}(x,y) \exp\left(- \int_{0}^{t} \lambda(s)\, ds\right).$$
As $T_{t} = S_{t}^{-1}$, the last inequality yields
$$d_{g}(x,y) \exp\left(\int_{0}^{t} \lambda(s)\, ds\right) \geq d_{g}(T_{t}(x), T_{t}(y));$$
thus $T_{t}$ is $\exp\left(\int_{0}^{t} \lambda(s)\, ds\right)$-Lipschitz, so the limiting map $T$ is $\exp\left(\int_{0}^{+\infty} \lambda(s)\, ds\right)$-Lipschitz.

Now we are ready to state the main result in this section and in this note.

\begin{theorem}
\label{thmmainresult}
Let $(M,g)$ be a complete and connected $d$-dimensional Riemannian manifold. Let $W \colon M \to \R$ be smooth and such that $\int_{M} e^{-W}d \mathrm{Vol} =1$. Consider the diffusion operator $\L=\Delta - \nabla W \cdot \nabla$. Let us assume that there exist constants $\rho_{1}, \rho_{2} > 0$ such that
	\begin{enumerate}[(i)]
	\item $\forall f \in \C^{\infty}_{\mathrm{c}}(M), \Gamma_{2}(f) \geq \rho_{1} \Gamma(f)$; and
	\item $\forall f \in \C^{\infty}_{\mathrm{c}}(M), \Gamma_{3}(f) \geq \rho_{2} \Gamma_{2}(f)$.
	\end{enumerate}
Let $V \colon M\to \R$ smooth and $K$-Lipschitz and define $d \mu = e^{-W}d \mathrm{Vol}$ and $d \nu = e^{-V} d\mu$. Then there exists a Lipschitz mapping $T \colon M \to M$ pushing forward $\mu$ towards $\nu$ which is $\exp\left(\sqrt{\frac{2 \pi}{\rho_{2}}}Ke^{\frac{K^{2}}{2 \rho_{1}}}\right)$-Lipschitz.
\end{theorem}

Before proceeding with the proof of the theorem, we will make a few remarks.

\begin{remark}
	\begin{itemize}
	\item First of all, we observe that the estimate for the Lipschitz constant given by Theorem \ref{thmmainresult} is intrinsecally independent on the dimension $d$ of the manifold $(M, g)$ ---as in Caffarelli's seminal result---, which is quite important for applications: in the next section, we will see how this result applies to transfer functional inequalities from the measure $\mu$ towards a log-Lipschitz perturbation of it, so if we have a dimension-free inequality for $\mu$, then the transported inequality for $\nu$ is dimension-free as well.
	\item Theorem 5 in \cite{fathi2023transportation} gives another estimate for the Lipschitz constant of $T$, in the manifold context as well. The assumptions of that result are similar to those of Theorem \ref{thmmainresult}, but instead of supposing that $\Gamma_{3} \geq \rho_{2} \Gamma_{2}$, it is assumed a uniform bound on $R$, the Riemann tensor of curvature (which fully characterizes curvature in manifolds). The estimates are of the same type, in the sense that both are of the form $O\left(\exp(\exp(K^{2}))\right)$, with $K$ the Lipschitz constant of the potential $V$, but the assumptions are different.
	\end{itemize}
\end{remark}

Before starting with the proof of Theorem \ref{thmmainresult}, we recall the following lemma, which again is a consequence of Herbst's argument (see \cite[p. 298]{MR4568961}).

\begin{lem}
\label{lemma5}
Let $(E, \delta, m)$ be a metric measure space and suppose that $m$ is a probability measure. If $m$ satisfies a log-Sobolev inequality with constant $\lambda_{\mathrm{LS}} > 0$, then for each $g \colon E\to \R$ $C$-Lipschitz, and any $-\infty < q < p < + \infty$,
$$\norm{e^{g}}_{L^{p}(m)} \leq \exp\left(C^{2}\frac{(p-q)}{2\lambda_{\mathrm{LS}}}\right) \norm{e^{g}}_{L^{q}(m)}.$$
\end{lem}

\begin{proof}[Proof (of Theorem \ref{thmmainresult})]
In view of what has been discussed in this section, and more particularly, in the light of Lemma \ref{lemboundonvtforlipschitz}, we only have to prove that we can find an integrable bound $\lambda$ for $\nabla^{2} \log\P_{t} f$; that is, $\nabla^{2} \log\P_{t} f(x)\preceq \lambda(t) g_{x}$. We note that for each $t > 0$, $x \in M$ and $Y \in T_{x}M$ with $\abs{Y}=1$:	
\begin{align*}
\nabla^{2} \log \P_{t}f(x) (Y, Y)  &= \frac{\nabla^{2} \P_{t} f(x) (Y, Y)}{\P_{t}f(x)} - \abs{ \nabla \log \P_{t}f(x)\cdot Y}^{2}\\
&\leq \frac{\nabla^{2} \P_{t} f(x) (Y, Y)}{\P_{t}f(x)}\\
&\leq \frac{\hs{\nabla^{2} \P_{t} f(x)}}{\P_{t}f(x)};
\end{align*}
that is,
\begin{equation}
\label{eqproofthmmain1}
\nabla^{2} \log \P_{t}f(x) (Y, Y) \leq \frac{\hs{\nabla^{2} \P_{t} f(x)}}{\P_{t}f(x)}.
\end{equation}

On the one hand, in the light of Bochner's formula (or more precisely (\ref{eqgamma2diffusion})), we note that $\hs{\nabla^{2} \P_{t} f} \leq \sqrt{\Gamma_{2}(\nabla^{2} \P_{t} f)}$. On the other hand, as both $\rho_{2}$ and $\rho_{1}$ are positive, then in the light of Remark \ref{rknonnegative}, we may use Proposition \ref{lemma4}; thus we have that $$\Gamma_{2}(\P_{t}f) \leq \frac1t e^{-\rho_{2}t} \P_{t}(\Gamma_{1}(f));$$ thus
\begin{equation}
\label{eqproofthmmain2}
	\hs{\nabla^{2} \P_{t} f}^{2} \leq \frac1t e^{-\rho_{2}t} \P_{t}(\Gamma_{1}(f)).
\end{equation}

Now, let us note that
	\begin{align*}
	\label{eqproofthmmain3}
	\P_{t}(\Gamma_{1}(f)) &= \P_{t}(\abs{\nabla f}^{2}) = \P_{t}(\abs{f \nabla V}^{2})\leq K^{2} \P_{t}(f^{2})\leq K^{2} \exp\left(K^{2}\frac{1-e^{-2\rho_{1} t}}{\rho_{1}}\right)(\P_{t}f)^{2}\\
	&\leq K^{2} \exp\left(\frac{K^{2}}{\rho_{1}}\right)(\P_{t}f)^{2}
	\end{align*}
where we employed Lemma \ref{lemma5}, for $p=2, q=1$ and the measure $m=\P_{t}^{*}\delta_{x}$, as it satisfies a log-Sobolev inequality with constant $\lambda_{\mathrm{LS}} = \frac{1-e^{-2\rho_{1} t}}{\rho_{1}}$ (condition (iv) in Theorem \ref{thmlocalinequalitiescurvature}); so if we mix this with both (\ref{eqproofthmmain1}) and (\ref{eqproofthmmain2}), then we finally get
	\begin{align*}
	\nabla^{2} \log \P_{t}f(x) (Y, Y) \leq  K e^{\frac{K^{2}}{2\rho_{1}}} \frac1{\sqrt{t}} e^{-\frac12 \rho_{2} t},
	\end{align*}
hence the Lipschitz estimate follows from the fact that	
	\begin{align*}
	\int_{0}^{+\infty} \frac1{\sqrt{t}} e^{-\frac12 \rho_{2}t} \, dt = \sqrt{\frac2{\rho_{2}}} \int_{0}^{+\infty} \frac1{\sqrt{s}} e^{-s}\, ds = \sqrt{\frac{2\pi}{\rho_{2}}}.
	\end{align*}
\end{proof}

%%%%% SECTION: Applications
\section{Applications}
\label{section5}

In this section we give different applications of Theorem \ref{thmmainresult}.

%%%%% SUBSECTION: Transfer of functional inequalities
\subsection{Transfer of functional inequalities}

As it is well known, having a Lipschitz map pushing forward a measure towards a log-Lipschitz perturbation of it allows us to transport functional inequalities from the source measure to the image one, as we did in the very beginning of the article in equation (\ref{eqfirstexample}), transporting Gaussian Poincar\'e's inequality. In particular, we recall Theorem \ref{thmpoincareandlogsobolev} which says that under $\mathrm{CD}(\rho, \infty)$ (for $\rho > 0$), then the measure $\mu$ satisfies both log-Sobolev and Poincar\'e's inequalities (both with constant $\frac1\rho$). More precisely, we have the following corollary.

\begin{cor}
\label{cortransfer}
Let $\L$ be a diffusion operator on $(M,g)$ with symmetric and ergodic measure $\mu$ satisfying the hypotheses of Theorem \ref{thmmainresult}. Then any $K$ log-Lipschitz perturbation of the measure $\mu$ satisfies both log-Sobolev and Poincar\'e's inequalities with constant $\frac{C_{K}}\rho$, where $C_{K} = \exp\left(2\sqrt{\frac{2 \pi}{\rho_{2}}}Ke^{\frac{K^{2}}{2 \rho_{1}}}\right)$.
\end{cor}

\begin{remark}
Transfer of functional inequalities for perturbations is a kind of result which has been present in literature for a long time (see for example \cite{MR1305076}, on the preservation of log-Sobolev's inequality for perturbations). Nevertheless, our method allows us to transfer not just inequalities in the form of Poincar\'e or log-Sobolev but isoperimetric and concentration inequalities.
\end{remark}

%%%%% SUBSECTION: The sphere
\subsection{The sphere}

In order to give a more concrete application, we use Theorem \ref{thmmainresult} in the context of the sphere $\mathbb{S}^{d}$. If we consider $\mathbb{S}^{d}$ as a weighted manifold with $\mu = \mathrm{Vol}$ and $\L=\Delta$, the Laplace-Beltrami operator on $\mathbb{S}^{d}$, then following Example \ref{explesphere}, we have that $\Gamma_{2} \geq (d-1) \Gamma$ and $\Gamma_{3} \geq (d-1) \Gamma_{2}$, so applying Theorem \ref{thmmainresult}, we obtain that each $V \colon \mathbb{S}^{d} \to \R$ 1-Lipschitz, then there exists a Lipschitz mapping $T$ pushing forward $\mathrm{Vol}$ towards $e^{-V}\, \mathrm{Vol}$ with Lipschitz constant $\exp\left(\sqrt{\frac{2\pi}{d-1}} e^{\frac1{2(d-1)}}\right)$.

Let us note that if we take $\mathbb{S}^{d}(\sqrt{d})$, the $d$-sphere of radius $\sqrt{d}$, and we set $\mu_{d}$ as the uniform measure on it, then its Ricci tensor is given by $\Ric_{\mathbb{S}^{d}(\sqrt{d})} = \frac{d-1}{d} g_{\mathbb{S}^{d}(\sqrt{d})}$, so it will satisfy $\Gamma_{2} \geq \frac{d-1}{d} \Gamma$ and $\Gamma_{3} \geq \frac{d-1}{d} \Gamma_{2}$; thus we can apply Theorem \ref{thmmainresult} to get a transport map with Lipschitz constant $\exp\left(\sqrt{2\pi\frac{d-1}{d}} e^{\frac1{2 \frac{d-1}{d}}}\right)$. Let us remark that asymptotically the Lipschitz constant converges to $\exp\left(\sqrt{2\pi} e^{\frac12}\right)$ as $d \to + \infty$.

\begin{remark}
It is clear that the Lipschitz constant obtained in this example is worser than the specific one provided in Theorem 3 of \cite{fathi2023transportation}, which stresses the constant curvature of $\mathbb{S}^{d}$. Nevertheless, we think that it illustrates the applicability of Theorem \ref{thmmainresult}; i.e., a diffusion generator $\L$ on a manifold satisfying $\Gamma_{2} \geq \rho_{1} \Gamma$ and $\Gamma_{3} \geq \rho_{2} \Gamma_{2}$ at the same time.
\end{remark}

%%%%% SUBSECTION: Laguerre generator
\subsection{Laguerre generator}

As we saw in Example \ref{explelaguerre}, for any $p \in (0,+\infty)^{d}$ such that for all $1 \leq i \leq d$, $p_{i} \geq \frac32$, then the Laguerre operator $\L_{p}$ is essentially self-adjoint and its associated semigroup conservative, so we may still apply Theorem \ref{thmmainresult}. Thus we have a globally Lipschitz map pushing forward the multivariate gamma distribution towards log-Lipschitz perturbations of it.

For the particular case when $d=1$ and $p=1$, where the ergodic measure $\mu$ corresponds to the exponential distribution on $(0,+\infty)$, following Remark \ref{rkexpdistrlaguerre} we have that $\L_{p}$ is essentially self-adjoint with respect to the class of functions $\C^{\infty}_{\mathrm{c,\, Neu}}((0,+\infty))$, but this makes no difference in the arguments, so we may apply as well Theorem \ref{thmmainresult}.

In the latter case, the operator $-\L_{1}$ has a discrete spectrum equal to $\N$; thus it satisfies Poincar\'e's inequality (with a better constant than the one given by $\mathrm{CD}(\frac12, \infty)$): for each $f \colon (0,+\infty) \to \R$ smooth,
$$\Var_{\mu_{1}}(f) \leq \int_{\R_{+}} x (f'(x))^{2} \, d\mu_{1}(x).$$
In particular, in the light of Corollary \ref{cortransfer}, it is possible to transfer this inequality for any log-Lipschitz perturbation of the exponential measure (for the metric generated by the Laguerre generator $\L_{1}$); i.e., if $\nu$ is the log-Lipschitz perturbation, then there exists $C>0$ such that for any smooth function $f$,
$$\Var_{\nu}(f) \leq C \int_{\R_{+}} y (f'(y))^{2} \, d\nu(y).$$

\begin{remark}
For the exponential measure there is a very rich theory in terms of the study of functional inequalities (see for example \cite{MR1440138} and \cite{MR2438906}); for example, it holds Poincar\'e's inequality with the usual carr\'e du champ operator, $f \mapsto (f')^{2}$: for any $f \colon (0,+\infty) \to \R$ smooth,
\begin{equation}
\label{eqpoincareexpclassical}
\Var_{\mu_{1}}(f) \leq 4 \int_{\R_{+}} (f')^{2}\, d\mu_{1}.
\end{equation}

At first glance, we are not able to use Corollary \ref{cortransfer} to extend the inequality (\ref{eqpoincareexpclassical}) towards log-Lipschitz perturbations (for the metric generated by the Laguerre generator $\L_{1}$, namely $x \mapsto \frac1x$), as the Poincar\'e inequality (\ref{eqpoincareexpclassical}) is stated in terms of the carr\'e du champ operator $f \mapsto (f')^{2}$, which is not compatible with the metric  $x \mapsto \frac1x$. Nevertheless, we observe that the second conclusion in Proposition \ref{propgrowthmongeexp} below allows us to extend the same kind of functional inequalities for its log-Lipschitz (with respect to the metric $x \mapsto \frac1x$) perturbations; that is, we get new functional inequalities, for the carr\'e du champ operator $f \mapsto (f')^{2}$, for these kind of perturbations $\nu$: i.e., there exists $C>0$ such that for any $f \colon (0,+\infty) \to \R$ smooth,
\begin{equation}
\label{eqpoincareexpclassicalperturb}
\Var_{\nu}(f) \leq C \int_{\R_{+}} (f')^{2}\, d\nu.
\end{equation}
\end{remark}

It is known that in dimension one, the diffusion transport map $T$ coincides with the Monge map. The following result provides a new estimate on the growth of Monge's map for $d=1$ and $p \geq \frac32$ or $p=\frac12$ (i.e. when the ergodic measure corresponds to the exponential measure). Let us recall that $T$ is non-decreasing (thus $T' \geq 0$) and in this context, positive, as $T \colon (0,+\infty) \to (0,+\infty)$.

\begin{proposition}
\label{propgrowthmongeexp}
Let $\mu_{p}$ be the gamma distribution on $(0,+\infty)$ and let $V\colon (0,+\infty) \to \R$ be a Lipschitz potential (for the metric $x \mapsto \frac1x$) and let $T$ be the Monge map pushing forward $\mu_{p}$ towards $e^{-V} \mu_{p}$. Then there exists a constant $C > 0$ such that for any $x >0$,
$$0 < T(x) \leq C x.$$
Moreover, $T$ is Lipschitz for the Euclidean metric on $(0,+\infty)$, that is, there exists $C'>0$ such that for any $x > 0$,
$$0 \leq T'(x) \leq C'.$$
\end{proposition}

\begin{proof}
Let us denote by $g$ the metric $x \mapsto \frac1x$. Using Theorem \ref{thmmainresult}, we deduce the existence of a constant $M>0$ such that $\sup_{x >0}\norm{T'(x)}_{\mathrm{op}} \leq M$. Therefore for each $x >0$,
	\begin{align*}
	M &\geq \norm{T'(x)}_{\mathrm{op}} = \sup_{\abs{v}_{g_{x}}=1}\abs{T'(x) \cdot v}_{g_{T(x)}} = \sup_{\abs{v}_{g_{x}}=1}\sqrt{\frac1{T(x)} \left(T'(x) v\right)^{2}}\\
	& =\sup_{\abs{v}_{g_{x}}=1}\sqrt{\frac{x}{T(x)} \left(T'(x) \right)^{2} \abs{v}^{2}_{g_{x}}}  = \sqrt{\frac{x}{T(x)}} T'(x),
	\end{align*}
which yields
$$\frac{d}{dx}\left(\sqrt{T(x)}\right) \leq M \frac1{2\sqrt{x}}.$$
Then if we integrate the last inequality and using the fact that $\lim_{y\to 0} T(y) = 0$, we obtain that for each $x > 0$, $\sqrt{T(x)} \leq M \sqrt{x}$, which yields the first claim in the Proposition for $C=M^{2}$

Now to bound $T'$, we just use the bound on $T$ and the first inequality, namely
$$M \geq \sqrt{\frac{x}{T(x)}} T'(x),$$
yielding the desired conclusion.
\end{proof}

\begin{remark}
Theorem 1.3 in \cite{MR4154935} provides a growth estimate for the derivative of the Monge map pushing forward the Gaussian measure $\gamma_{d}$ onto a log-concave measure on $\R^{d}$ which moreover is log-Lipschitz (satisfying certain technical bounds for the Hessian of its log-density): more precisely, $\norm{\nabla T (x)}_{\mathrm{op}} = O(1 + \abs{x}^{2})$ as $\abs{x} \to \infty$. Proposition \ref{propgrowthmongeexp} states that, for the family of gamma distributions the growth for the derivative of Monge's map pushing forward the measure towards a log-Lipschitz perturbation is $O(1)$ as $x \to \infty$ for the Euclidean metric.
\end{remark}

%\nocite{*}
    %\footnotesize{\bibliographystyle{plain}\bibliography{biblio}}
    \bibliographystyle{alpha}
%\addcontentsline{toc}{section}{References}
\bibliography{biblio}
%\printbibheading
%\printbibliography

\end{document}